\documentclass[a4paper]{amsart}

\usepackage{a4wide}

\usepackage[T1]{fontenc}
\usepackage[utf8]{inputenc}
\usepackage{subfigure}
\usepackage{hyperref}
\usepackage{cite}
\usepackage{amsmath}
\usepackage{amsthm}

\usepackage{tikz}

\usetikzlibrary{automata}

\newcommand{\M}{\mathcal{M}}
\newcommand{\B}{\mathcal{B}}

\newcommand{\E}{\mathbb{E}}

\newcommand{\Prob}{\mathbb{P}}

\newcommand{\hn}{h_{\mathsf{NAF}}}
\newcommand{\hg}{h_{\mathsf{GRAY}}}
\newcommand{\ro}{r_{\mathsf{OPT}}}
\DeclareMathOperator{\rank}{rank}
\newcommand{\bff}{\boldsymbol{f}}
\newcommand{\bfu}{\boldsymbol{u}}
\newcommand{\bfv}{\boldsymbol{v}}
\newcommand{\bfx}{\boldsymbol{x}}
\newcommand{\bfy}{\boldsymbol{y}}
\newcommand{\bfdelta}{\boldsymbol{\delta}}
\newcommand{\bfb}{\boldsymbol{b}}
\newcommand{\eps}{\varepsilon}

\newif\ifdetails
\detailstrue
\newcommand{\DETAIL}[1]%
{\ifdetails\par\fbox{\begin{minipage}{0.9\linewidth}\textit{Detail:}
      #1\end{minipage}}\par\fi}
\newcommand{\TODO}[1]%
{\ifdetails\par\fbox{\begin{minipage}{0.9\linewidth}\textbf{TODO:}
      #1\end{minipage}}\par\fi}

\newtheorem{lemma}{Lemma}
\newtheorem{prop}[lemma]{Proposition}
\newtheorem{theorem}[lemma]{Theorem}
\newtheorem{cor}[lemma]{Corollary}
\theoremstyle{remark}
\newtheorem{example}{Example}
\newtheorem*{remark}{Remark}
\newtheorem*{defi}{Definition}

\title{On $q$-Quasiadditive and $q$-Quasimultiplicative Functions}

\author{Sara Kropf}
\address{Institut f\"ur Mathematik \\
Alpen-Adria-Universit\"at Klagenfurt \\
Austria \\
and Institute of Statistical Science \\
Academia Sinica, Taipei \\
Taiwan
}
\email{sara.kropf@aau.at \textnormal{and} sarakropf@stat.sinica.edu.tw}
\author{Stephan Wagner}
\address{Department of Mathematical Sciences \\
Stellenbosch University \\
South Africa}
\email{swagner@sun.ac.za}

\keywords{$q$-additive function, $q$-quasiadditive function, $q$-regular function, central limit theorem}

\thanks{The first author is supported by the Austrian Science Fund (FWF): P~24644-N26. The second author is supported by the National Research Foundation of South Africa under grant number 96236. The authors were also supported by the Karl Popper Kolleg ``Modeling--Simulation--Optimization'' funded by the Alpen-Adria-Universit\"at
Klagenfurt and by the Carinthian Economic Promotion Fund (KWF). Part
of this paper was written while the second author was a Karl Popper
Fellow at the Mathematics Institute in Klagenfurt. He would like to
thank the institute for the hospitality received.\\
An extended abstract was presented at the 27th International Meeting on Probabilistic, Combinatorial, and Asymptotic Methods for the Analysis of Algorithms, Kraków, 4--8 July, 2016, see \cite{Kropf-Wagner:2016:quasiad-funct}.}

\begin{document}
\maketitle

\begin{abstract}
In this paper, we introduce the notion of \begin{math}q\end{math}-quasiadditivity of
arithmetic functions, as well as the related concept of
\begin{math}q\end{math}-quasimultiplicativity, which generalise strong \begin{math}q\end{math}-additivity
and -multiplicativity, respectively. We show that there are many
natural examples for these concepts, which are characterised by
functional equations of the form \begin{math}f(q^{k+r}a + b) = f(a) + f(b)\end{math} or
\begin{math}f(q^{k+r}a + b) = f(a) f(b)\end{math} for all \begin{math}b < q^k\end{math} and a fixed parameter \begin{math}r\end{math}.
In addition to some elementary properties of \begin{math}q\end{math}-quasiadditive and \begin{math}q\end{math}-quasimultiplicative functions, we prove characterisations of \begin{math}q\end{math}-quasiadditivity and \begin{math}q\end{math}-quasimultiplicativity for the special class of \begin{math}q\end{math}-regular functions. The final main result provides a general central limit theorem that includes both classical and new examples as corollaries.
\end{abstract}

\section{Introduction}

Arithmetic functions based on the digital expansion in some base \begin{math}q\end{math}
have a long history (see, e.g., \cite{Bellman-Shapiro:1948,Gelfond:1968:sur,Delange:1972:q-add-q-mult,Delange:1975:chiffres,Cateland:digital-seq,Bassily-Katai:1995:distr,Drmota:2000})
 The notion of a \begin{math}q\end{math}-\emph{additive} function is due to \cite{Gelfond:1968:sur}: an arithmetic function (defined on nonnegative integers) is called \begin{math}q\end{math}-additive if
\begin{equation*}f(q^k a + b) = f(q^k a) + f(b)\end{equation*}
whenever \begin{math}0 \leq b < q^k\end{math}. A stronger version of this concept is \emph{strong} (or \emph{complete}) \begin{math}q\end{math}-additivity: a function \begin{math}f\end{math} is said to be strongly \begin{math}q\end{math}-additive if we even have
\begin{equation*}f(q^k a + b) = f(a) + f(b)\end{equation*}
whenever \begin{math}0 \leq b < q^k\end{math}. The class of (strongly) \begin{math}q\end{math}-\emph{multiplicative} functions is defined in an analogous fashion.
Loosely speaking, (strong) \begin{math}q\end{math}-additivity of a function means that
it can be evaluated by breaking up the base-\begin{math}q\end{math} expansion. Typical
examples of strongly \begin{math}q\end{math}-additive functions are the \begin{math}q\end{math}-ary sum of
digits and the number of occurrences of a specified nonzero digit.

There are, however, many simple and natural functions based on the \begin{math}q\end{math}-ary expansion that are not \begin{math}q\end{math}-additive. A very basic example of this kind are \emph{block counts}: the number of occurrences of a certain block of digits in the \begin{math}q\end{math}-ary expansion. This and other examples provide the motivation for the present paper, in which we define and study a larger class of functions with comparable properties.

\begin{defi}
An arithmetic function (a function defined on the set of nonnegative integers) is called \begin{math}q\end{math}-\emph{quasiadditive} if there exists some nonnegative integer \begin{math}r\end{math} such that
\begin{equation}\label{eq:q-add}
f(q^{k+r}a + b) = f(a) + f(b)
\end{equation}
whenever \begin{math}0 \leq b < q^k\end{math}. Likewise, \begin{math}f\end{math} is said to be \begin{math}q\end{math}-\emph{quasimultiplicative} if it satisfies the identity
\begin{equation}\label{eq:q-mult}
f(q^{k+r}a + b) = f(a)f(b)
\end{equation}
for some fixed nonnegative integer \begin{math}r\end{math} whenever \begin{math}0 \leq b < q^k\end{math}.
\end{defi}

We remark that the special case \begin{math}r = 0\end{math} is exactly strong
\begin{math}q\end{math}-additivity, so strictly speaking the term ``strongly
\begin{math}q\end{math}-quasiadditive function'' might be more appropriate. However, since
we are not considering a weaker version (for which natural examples
seem to be much harder to find), we do not make a distinction. As a further caveat, we remark that the term ``quasiadditivity'' has also been used in \cite{allouche:1993} for a related, but slightly weaker condition.

In the
following section, we present a variety of examples of
\begin{math}q\end{math}-quasiadditive and \begin{math}q\end{math}-quasimultipli\-cative functions. 
In Section~\ref{sec:elem-properties}, we give some general properties of
such functions. Since most of our examples also belong to the related class of \begin{math}q\end{math}-regular
functions, we discuss the connection in Section~\ref{sec:q-regular}.
Finally, we prove a general central limit theorem for \begin{math}q\end{math}-quasiadditive and -multiplicative functions that contains both old
and new examples as special cases.

\section{Examples of $q$-quasiadditive and $q$-quasimultiplicative  functions}
\label{sec:exampl-q-quasiadd}
Let us now back up the abstract concept of \begin{math}q\end{math}-quasiadditivity by some concrete examples.

\subsection*{Block counts}

As mentioned in the introduction, the number of occurrences of a fixed
digit is a typical example of a \begin{math}q\end{math}-additive function. However, the
number of occurrences of a given block \begin{math}B = \epsilon_1\epsilon_2
\cdots \epsilon_{\ell}\end{math} of digits in the expansion of a nonnegative
integer \begin{math}n\end{math}, which we denote by \begin{math}c_B(n)\end{math}, does not represent a
\begin{math}q\end{math}-additive function. The reason is simple: the \begin{math}q\end{math}-ary expansion of \begin{math}q^ka + b\end{math} is obtained by joining the expansions of \begin{math}a\end{math} and \begin{math}b\end{math}, so occurrences of \begin{math}B\end{math} in \begin{math}a\end{math} and occurrences of \begin{math}B\end{math} in \begin{math}b\end{math} are counted by \begin{math}c_B(a) + c_{B}(b)\end{math}, but occurrences that involve digits of both \begin{math}a\end{math} and \begin{math}b\end{math} are not.

However, if \begin{math}B\end{math} is a block different from \begin{math}00\cdots0\end{math}, then \begin{math}c_B\end{math} is \begin{math}q\end{math}-quasiadditive: note that the representation of \begin{math}q^{k+\ell} a + b\end{math} is of the form
\begin{equation*}\underbrace{a_1 a_2 \cdots a_{\mu}}_{\text{expansion of } a} \underbrace{0 0 \cdots 0_{\vphantom{\mu}}}_{\ell \text{ zeros}} \underbrace{b_1 b_2 \cdots {b_{\nu}}_{\vphantom{\mu}}}_{\text{expansion of } b}\end{equation*}
whenever \begin{math}0 \leq b < q^k\end{math}, so occurrences of the block \begin{math}B\end{math} have to belong to either \begin{math}a\end{math} or \begin{math}b\end{math} only, implying that
\begin{math}c_B(q^{k+\ell} a + b) = c_B(a) + c_B(b)\end{math},
with one small caveat: if the block starts and/or ends with a sequence
of zeros, then the count needs to be adjusted by assuming the digital
expansion of a nonnegative integer to be padded with zeros on the left
and on the right. 

For example, let \begin{math}B\end{math} be the block \begin{math}0101\end{math} in base \begin{math}2\end{math}. The binary representations of \begin{math}469\end{math} and \begin{math}22\end{math} are \begin{math}111010101\end{math} and \begin{math}10110\end{math}, respectively, so we have \begin{math}c_B(469) = 2\end{math} and \begin{math}c_B(22) = 1\end{math} (note the occurrence of \begin{math}0101\end{math} at the beginning of \begin{math}10110\end{math} if we assume the expansion to be padded with zeros), as well as
\begin{equation*}c_B(240150) = c_B(2^9 \cdot 469 + 22) = c_B(469) + c_B(22) = 3.\end{equation*}
Indeed, the block \begin{math}B\end{math} occurs three times in the expansion of \begin{math}240150\end{math}, which is \begin{math}111010101000010110\end{math}.

\subsection*{The number of runs and the Gray code}

The number of ones in the Gray code of a nonnegative integer \begin{math}n\end{math},
which we denote by \begin{math}\hg(n)\end{math}, is also equal to the number of runs
(maximal sequences of consecutive identical digits) in the binary
representations of \begin{math}n\end{math} (counting the number of runs in the
representation of \begin{math}0\end{math} as \begin{math}0\end{math}); the sequence defined by \begin{math}\hg(n)\end{math} is
\href{http://oeis.org/A005811}{A005811} in Sloane's On-Line Encyclopedia of Integer Sequences
\cite{OEIS:2016}. An analysis of its expected value is performed in \cite{Flajolet-Ramshaw:1980:gray}. The function \begin{math}\hg\end{math} is \begin{math}2\end{math}-quasiadditive up to some minor
modification: set \begin{math}f(n) = \hg(n)\end{math} if \begin{math}n\end{math} is even and \begin{math}f(n) = \hg(n)
+ 1\end{math} if \begin{math}n\end{math} is odd. The new function \begin{math}f\end{math} can be interpreted as the
total number of occurrences of the two blocks \begin{math}01\end{math} and \begin{math}10\end{math} in the
binary expansion (considering binary expansions to be padded with zeros at both ends), so the argument of the previous example applies again and shows that \begin{math}f\end{math} is \begin{math}2\end{math}-quasiadditive.

\subsection*{The nonadjacent form and its Hamming weight}

The nonadjacent form (NAF) of a nonnegative integer is the unique
base-\begin{math}2\end{math} representation with digits \begin{math}0,1,-1\end{math} (\begin{math}-1\end{math} is usually
represented as \begin{math}\overline{1}\end{math} in this context) and the additional
requirement that there may not be two adjacent nonzero digits, see
\cite{Reitwiesner:1960}. For example, the NAF of \begin{math}27\end{math} is
\begin{math}100\overline{1}0\overline{1}\end{math}. It is well known that the NAF always
has minimum Hamming weight (i.e., the number of nonzero digits) among all
possible binary representations with this particular digit set,
although it may not be unique with this property (compare, e.g.,
\cite{Reitwiesner:1960} with \cite{Joye-Yen:2000:optim-left}).

The Hamming weight \begin{math}\hn\end{math} of the nonadjacent form has been analysed in
some detail \cite{Thuswaldner:1999,Heuberger-Kropf:2013:analy}, and it is also an example of a \begin{math}2\end{math}-quasiadditive function. It is not difficult to see that \begin{math}\hn\end{math} is characterised by the recursions
\begin{equation*}\hn(2n) = \hn(n), \qquad \hn(4n+1) = \hn(n) + 1, \qquad \hn(4n-1) = \hn(n) + 1\end{equation*}
together with the initial value \begin{math}\hn(0) = 0\end{math}. The identity
\begin{equation*}\hn(2^{k+2}a + b) = \hn(a) + \hn(b)\end{equation*}
can be proved by induction. In Section~\ref{sec:q-regular}, this example will be generalised and put into a larger context.

\subsection*{The number of optimal $\{0,1,-1\}$-representations}

As mentioned above, the NAF may not be the only representation with minimum Hamming weight among all possible binary representations with digits \begin{math}0,1,-1\end{math}. The number of optimal representations of a given nonnegative integer \begin{math}n\end{math} is therefore a quantity of interest in its own right. Its average over intervals of the form \begin{math}[0,N)\end{math} was studied by Grabner and Heuberger \cite{Grabner-Heuberger:2006:Number-Optimal}, who also proved that the number \begin{math}\ro(n)\end{math} of optimal representations of \begin{math}n\end{math} can be obtained in the following way:

\begin{lemma}[Grabner--Heuberger \cite{Grabner-Heuberger:2006:Number-Optimal}]\label{lemma:opt-representations-recursion}
Let sequences \begin{math}u_i\end{math} (\begin{math}i=1,2,\ldots,5\end{math}) be given recursively by
\begin{equation*}u_1(0) = u_2(0) = \cdots = u_5(0) = 1, \qquad u_1(1) = u_2(1) = 1,\ u_3(1) = u_4(1) = u_5(1) = 0,\end{equation*}
and
\begin{align*}
u_1(2n) = u_1(n), \qquad & u_1(2n+1) = u_2(n) + u_4(n+1), \\
u_2(2n) = u_1(n), \qquad & u_2(2n+1) = u_3(n), \\
u_3(2n) = u_2(n), \qquad & u_3(2n+1) = 0, \\
u_4(2n) = u_1(n), \qquad & u_4(2n+1) = u_5(n+1), \\
u_5(2n) = u_4(n), \qquad & u_5(2n+1) = 0.
\end{align*}
The number \begin{math}\ro(n)\end{math} of optimal representations of \begin{math}n\end{math} is equal to \begin{math}u_1(n)\end{math}.
\end{lemma}

A straightforward calculation shows that
\begin{equation}\label{eq:8n_a}
\begin{aligned}
&u_1(8n) = u_2(8n) = \cdots = u_5(8n) = u_1(8n+1) = u_2(8n+1) = u_1(n),\\
&u_3(8n+1) = u_4(8n+1) = u_5(8n+1) = 0.
\end{aligned}
\end{equation}
This gives us the following result:

\begin{lemma}\label{lem:optrep}
The number of optimal \begin{math}\{0,1,-1\}\end{math}-representations of a nonnegative integer is a \begin{math}2\end{math}-quasimulti\-plicative function. Specifically, for any three nonnegative integers \begin{math}a,b,k\end{math} with \begin{math}b < 2^k\end{math}, we have
\begin{equation*}\ro(2^{k+3}a + b) = \ro(a)\ro(b).\end{equation*}
\end{lemma}

\begin{proof} We will prove a somewhat stronger statement by induction on \begin{math}t\end{math}: write
\begin{equation*}\mathbf{u}(n) = (u_1(n),u_2(n),u_3(n),u_4(n),u_5(n))^t.\end{equation*}
We show that
\begin{equation*}\mathbf{u}(2^{k+3}a + b) = \ro(a) \mathbf{u}(b)\end{equation*}
and
\begin{equation*}\mathbf{u}(2^{k+3}a + b+1) = \ro(a) \mathbf{u}(b+1)\end{equation*}
for all \begin{math}a,b,k\end{math} satisfying the conditions of the lemma, from which the desired result follows by considering the first entry of the vector \begin{math}\mathbf{u}(2^{k+3}a+b)\end{math}. Note first that both identities are clearly true for \begin{math}k=0\end{math} in view of~\eqref{eq:8n_a}. For the induction step, we distinguish two cases: if \begin{math}b\end{math} is even, we have
\begin{align*}
\mathbf{u}(2^{k+3}a + b) &=
\begin{pmatrix}
1 & 0 & 0 & 0 & 0 \\ 
1 & 0 & 0 & 0 & 0 \\
0 & 1 & 0 & 0 & 0 \\ 
1 & 0 & 0 & 0 & 0 \\ 
0 & 0 & 0 & 1 & 0
\end{pmatrix} \cdot \mathbf{u}(2^{k+2}a + b/2) \\
&= \begin{pmatrix}
1 & 0 & 0 & 0 & 0 \\ 
1 & 0 & 0 & 0 & 0 \\
0 & 1 & 0 & 0 & 0 \\ 
1 & 0 & 0 & 0 & 0 \\ 
0 & 0 & 0 & 1 & 0
\end{pmatrix} \cdot \ro(a) \mathbf{u}(b/2) \\
&= \ro(a) \mathbf{u}(b)
\end{align*}
by the induction hypothesis, as well as
\begin{align*}
\mathbf{u}(2^{k+3}a + b+1) &=
\begin{pmatrix}
0 & 1 & 0 & 0 & 0 \\ 
0 & 0 & 1 & 0 & 0 \\
0 & 0 & 0 & 0 & 0 \\ 
0 & 0 & 0 & 0 & 0 \\ 
0 & 0 & 0 & 0 & 0
\end{pmatrix} \cdot \mathbf{u}(2^{k+2}a + b/2) + \begin{pmatrix}
0 & 0 & 0 & 1 & 0 \\ 
0 & 0 & 0 & 0 & 0 \\
0 & 0 & 0 & 0 & 0 \\ 
0 & 0 & 0 & 0 & 1 \\ 
0 & 0 & 0 & 0 & 0
\end{pmatrix} \cdot \mathbf{u}(2^{k+2}a + b/2+1) \\
&= \begin{pmatrix}
0 & 1 & 0 & 0 & 0 \\ 
0 & 0 & 1 & 0 & 0 \\
0 & 0 & 0 & 0 & 0 \\ 
0 & 0 & 0 & 0 & 0 \\ 
0 & 0 & 0 & 0 & 0
\end{pmatrix} \cdot \ro(a)\mathbf{u}(b/2) + \begin{pmatrix}
0 & 0 & 0 & 1 & 0 \\ 
0 & 0 & 0 & 0 & 0 \\
0 & 0 & 0 & 0 & 0 \\ 
0 & 0 & 0 & 0 & 1 \\ 
0 & 0 & 0 & 0 & 0
\end{pmatrix} \cdot \ro(a)\mathbf{u}(b/2+1) \\
&= \ro(a)\mathbf{u}(b+1).
\end{align*}
The case that \begin{math}b\end{math} is odd is treated in an analogous fashion. 
\end{proof}

In Section~\ref{sec:q-regular}, we will show that this is also an instance of a more general phenomenon.

\subsection*{The run length transform and cellular automata}

The \emph{run length transform} of a sequence is defined in a recent paper of Sloane \cite{Sloane:number-on}: it is based on the binary representation, but could in principle also be generalised to other bases. Given a sequence \begin{math}s_1,s_2,\ldots\end{math}, its run length transform is obtained by the rule
\begin{equation*}t(n) = \prod_{i \in \mathcal{L}(n)} s_i,\end{equation*}
where \begin{math}\mathcal{L}(n)\end{math} is the multiset of run lengths of \begin{math}n\end{math} (lengths
of blocks of consecutive ones in the binary representation). For
example, the binary expansion of \begin{math}1910\end{math} is \begin{math}11101110110\end{math}, so the
multiset \begin{math}\mathcal{L}(n)\end{math} of run lengths would be \begin{math}\{3,3,2\}\end{math}, giving
\begin{math}t(1910) = s_2 s_3^2\end{math}.

A typical example is obtained for the sequence of Jacobsthal numbers given by the formula \begin{math}s_n = \frac13 (2^{n+2} - (-1)^n)\end{math}. The associated run length transform \begin{math}t_n\end{math} (sequence \href{http://oeis.org/A071053}{A071053} in the OEIS \cite{OEIS:2016}) counts the number of odd coefficients in the expansion of \begin{math}(1+x+x^2)^n\end{math}, and it can also be interpreted as the number of active cells at the \begin{math}n\end{math}-th generation of a certain cellular automaton. Further examples stemming from cellular automata can be found in Sloane's paper \cite{Sloane:number-on}.

The argument that proved \begin{math}q\end{math}-quasiadditivity of block counts also applies here, and indeed it is easy to see that the identity
\begin{equation*}t(2^{k+1}a + b) = t(a)t(b),\end{equation*}
where \begin{math}0 \leq b < 2^k\end{math}, holds for the run length transform of any sequence, meaning that any such transform is \begin{math}2\end{math}-quasimultiplicative. In fact, it is not difficult to show that every \begin{math}2\end{math}-quasimultiplicative function with parameter \begin{math}r=1\end{math} is the run length transform of some sequence.

\section{Elementary properties}
\label{sec:elem-properties}
Now that we have gathered some motivating examples for the concepts of \begin{math}q\end{math}-quasiadditivity and \begin{math}q\end{math}-quasi\-multiplicativity, let us present some simple results about functions with these properties. First of all, let us state an obvious relation between \begin{math}q\end{math}-quasiadditive and \begin{math}q\end{math}-quasimultiplicative functions:

\begin{prop}\label{prop:trivial}
If a function \begin{math}f\end{math} is \begin{math}q\end{math}-quasiadditive, then the function defined by \begin{math}g(n) = c^{f(n)}\end{math} for some positive constant \begin{math}c\end{math} is \begin{math}q\end{math}-quasimultiplicative. Conversely, if \begin{math}f\end{math} is a \begin{math}q\end{math}-quasimultiplicative function that only takes positive values, then the function defined by \begin{math}g(n) = \log_c f(n)\end{math} for some positive constant \begin{math}c \neq 1\end{math} is \begin{math}q\end{math}-quasiadditive.
\end{prop}

The next proposition deals with the parameter \begin{math}r\end{math} in the definition of a \begin{math}q\end{math}-quasiadditive function:

\begin{prop}
If the arithmetic function \begin{math}f\end{math} satisfies
\begin{equation*}f(q^{k+r}a + b) = f(a) + f(b)\end{equation*}
for some fixed nonnegative integer \begin{math}r\end{math} whenever \begin{math}0 \leq b < q^k\end{math}, then it also satisfies
\begin{equation*}f(q^{k+s}a + b) = f(a) + f(b)\end{equation*}
for all nonnegative integers \begin{math}s \geq r\end{math} whenever \begin{math}0 \leq b < q^k\end{math}.
\end{prop}

\begin{proof}
If \begin{math}a,b\end{math} are nonnegative integers with \begin{math}0 \leq b < q^k\end{math}, then clearly also \begin{math}0 \leq b < q^{k+s-r}\end{math} if \begin{math}s \geq r\end{math}, and thus
\begin{equation*}f(q^{k+s}a + b) = f(q^{(k+s-r)+r}a + b) = f(a) + f(b).\end{equation*}
\end{proof}

\begin{cor}\label{cor:lin_comb}
If two arithmetic functions \begin{math}f\end{math} and \begin{math}g\end{math} are \begin{math}q\end{math}-quasiadditive functions, then so is any linear combination \begin{math}\alpha f + \beta g\end{math} of the two.
\end{cor}
\begin{proof}
In view of the previous proposition, we may assume the parameter \begin{math}r\end{math} in~\eqref{eq:q-add} to be the same for both functions. The statement follows immediately.
\end{proof}

Finally, we observe that \begin{math}q\end{math}-quasiadditive and \begin{math}q\end{math}-quasimultiplicative functions can be computed by breaking the \begin{math}q\end{math}-ary expansion into pieces.

\begin{lemma}\label{lem:simplefacts}
If \begin{math}f\end{math} is a \begin{math}q\end{math}-quasiadditive (\begin{math}q\end{math}-quasimultiplicative) function, then
\begin{itemize}
\item \begin{math}f(0) = 0\end{math} (\begin{math}f(0) = 1\end{math}, respectively, unless \begin{math}f\end{math} is identically \begin{math}0\end{math}),
\item \begin{math}f(qa) = f(a)\end{math} for all nonnegative integers \begin{math}a\end{math}.
\end{itemize}
\end{lemma}

\begin{proof}
Assume first that \begin{math}f\end{math} is \begin{math}q\end{math}-quasiadditive. Setting \begin{math}a = b = 0\end{math} in the defining functional equation~\eqref{eq:q-add}, we obtain
\begin{equation*}f(0) = f(0) + f(0),\end{equation*}
and the first statement follows. Setting \begin{math}b = 0\end{math} while \begin{math}a\end{math} is arbitrary, we now find that
\begin{equation*}f(q^{k+r}a) = f(a)\end{equation*}
for all \begin{math}k \geq 0\end{math}. In particular, this also means that
\begin{equation*}f(a) = f(q^{r+1}a) = f(q^r \cdot qa) = f(qa),\end{equation*}
which proves the second statement. For \begin{math}q\end{math}-quasimultiplicative
functions, the proof is analogous (and one can also use
Proposition~\ref{prop:trivial} for positive functions).
\end{proof}

\begin{prop}\label{prop:split}
Suppose that the function \begin{math}f\end{math} is \begin{math}q\end{math}-quasiadditive with parameter \begin{math}r\end{math}, i.e., \begin{math}f(q^{k+r}a + b) = f(a) + f(b)\end{math} whenever \begin{math}0 \leq b < q^k\end{math}. Going from left to right, split the $q$-ary expansion of \begin{math}n\end{math} into blocks by inserting breaks after each run of \begin{math}r\end{math} or more zeros. If these blocks are the $q$-ary representations of \begin{math}n_1,n_2,\ldots,n_{\ell}\end{math}, then we have
\begin{equation*}f(n) = f(n_1) + f(n_2) + \cdots + f(n_{\ell}).\end{equation*}
Moreover, if \begin{math}m_i\end{math} is
the greatest divisor
of \begin{math}n_i\end{math} which
is not divisible by \begin{math}q\end{math} for $i=1,\ldots,\ell$, then
\begin{equation*}f(n) = f(m_1) + f(m_2) + \cdots + f(m_{\ell}).\end{equation*}
Analogous statements hold for \begin{math}q\end{math}-quasimultiplicative functions, with sums replaced by products.
\end{prop}

\begin{proof}
This is obtained by a straightforward induction on \begin{math}\ell\end{math} together with the fact that \begin{math}f(q^{h} a) = f(a)\end{math}, which follows from the previous lemma.
\end{proof}

\begin{example}
Recall that the Hamming weight of the NAF (which is the minimum Hamming weight of a \begin{math}\{0,1,-1\}\end{math}-representation) is \begin{math}2\end{math}-quasiadditive with parameter \begin{math}r=2\end{math}. To determine \begin{math}\hn(314\,159\,265)\end{math}, we split the binary representation, which is
\begin{math}10010101110011011000010100001,\end{math}
into blocks by inserting breaks after each run of at least two zeros:
\begin{equation*}100|101011100|110110000|1010000|1.\end{equation*}
The numbers \begin{math}n_1,n_2,\ldots,n_{\ell}\end{math} in the statement of the proposition are now \begin{math}4,348,432,80,1\end{math} respectively, and the numbers \begin{math}m_1,m_2,\ldots,m_{\ell}\end{math} are therefore \begin{math}1,87,27,5,1\end{math}. Now we use the values \begin{math}\hn(1) = 1\end{math}, \begin{math}\hn(5) = 2\end{math}, \begin{math}\hn(27) = 3\end{math} and \begin{math}\hn(87) = 4\end{math} to obtain
\begin{equation*}\hn(314\,159\,265) = 2\hn(1) + \hn(5) + \hn(27) + \hn(87) = 11.\end{equation*}
\end{example}

\begin{example}
In the same way, we consider the number of optimal representations \begin{math}\ro\end{math}, which is \begin{math}2\end{math}-quasimultiplicative with parameter \begin{math}r=3\end{math}. Consider for instance the binary representation of \begin{math}204\,280\,974\end{math}, namely 
\begin{math}1100001011010001010010001110\end{math}.
We split into blocks:
\begin{equation*}110000|101101000|101001000|1110.\end{equation*}
The four blocks correspond to the numbers \begin{math}48 = 16 \cdot 3\end{math}, \begin{math}360 = 8 \cdot 45\end{math}, \begin{math}328 = 8 \cdot 41\end{math} and \begin{math}14 = 2 \cdot 7\end{math}. Since \begin{math}\ro(3) = 2\end{math}, \begin{math}\ro(45) = 5\end{math}, \begin{math}\ro(41) = 1\end{math} and \begin{math}\ro(7) = 1\end{math}, we obtain
\begin{math}\ro(204\,280\,974) = 10\end{math}.
\end{example}

\section{$q$-Regular functions}\label{sec:q-regular}
In this section, we introduce \begin{math}q\end{math}-regular functions and examine the
connection to our concepts. See~\cite{Allouche-Shallit:2003:autom} for
more background on \begin{math}q\end{math}-regular functions and sequences.

A function \begin{math}f\end{math} is
\emph{\begin{math}q\end{math}-regular} if it can be expressed as \begin{math}f=\bfu^{t}\bff\end{math} for a vector \begin{math}\bfu\end{math}
and a vector-valued function \begin{math}\bff\end{math}, and there are matrices \begin{math}M_{i}\end{math}, \begin{math}0\leq i<q\end{math}, satisfying
\begin{equation}\label{eq:q-regular-recursive}
  \bff(qn+i)=M_{i}\bff(n)
\end{equation}
for \begin{math}0\leq i<q\end{math}, \begin{math}qn+i>0\end{math}. We set \begin{math}\bfv=\bff(0)\end{math}.

Equivalently, a function \begin{math}f\end{math} is \begin{math}q\end{math}-regular if and only if \begin{math}f\end{math} can be written as
\begin{equation}
  \label{eq:q-regular}
  f(n)=\bfu^{t} \prod_{i=0}^{L} M_{n_{i}}\bfv
\end{equation}
where \begin{math}n_{L}\cdots n_{0}\end{math} is the \begin{math}q\end{math}-ary expansion of \begin{math}n\end{math}.

The notion of \begin{math}q\end{math}-regular functions is a generalisation of
\begin{math}q\end{math}-additive and \begin{math}q\end{math}-multiplicative functions. However, we emphasise that \begin{math}q\end{math}-quasiadditive and \begin{math}q\end{math}-quasimultiplicative functions are not
necessarily \begin{math}q\end{math}-regular: a \begin{math}q\end{math}-regular sequence can always be bounded
by \begin{math}O(n^{c})\end{math} for a constant \begin{math}c\end{math}, see~\cite[Thm.\
16.3.1]{Allouche-Shallit:2003:autom}. In our setting however, the values of \begin{math}f(n)\end{math} can be chosen arbitrarily for those \begin{math}n\end{math} whose \begin{math}q\end{math}-ary expansion does not contain \begin{math}0^{r}\end{math}. Therefore a \begin{math}q\end{math}-quasiadditive or -multiplicative function can grow arbitrarily fast.

We call \begin{math}(\bfu, (M_{i})_{0\leq i<q}, \bfv)\end{math} a \emph{linear representation} of the
\begin{math}q\end{math}-regular function \begin{math}f\end{math}. Such a linear representation is called
\emph{zero-insensitive} if \begin{math}M_{0}\bfv=\bfv\end{math}, meaning that in
\eqref{eq:q-regular}, leading zeros in the \begin{math}q\end{math}-ary expansion of \begin{math}n\end{math} do
not change anything. We call a linear representation \emph{minimal} if the dimension
of the matrices \begin{math}M_{i}\end{math} is minimal among all linear representations of \begin{math}f\end{math}.

 Following \cite{Dumas:2014:asymp}, every \begin{math}q\end{math}-regular function has a
 zero-insensitive minimal linear representation.

\subsection{When is a $q$-regular function $q$-quasimultiplicative?}
We now give a characterisation of \begin{math}q\end{math}-regular functions that are \begin{math}q\end{math}-quasimultiplicative. 

\begin{theorem}\label{theorem:reg-mult}
  Let \begin{math}f\end{math} be a \begin{math}q\end{math}-regular sequence with zero-insensitive minimal linear representation~\eqref{eq:q-regular}. Then the following
  two assertions are equivalent:
  \begin{itemize}
  \item The sequence \begin{math}f\end{math} is \begin{math}q\end{math}-quasimultiplicative with parameter
    \begin{math}r\end{math}.
    \item    \begin{math}M_{0}^{r}=\bfv\bfu^{t}\end{math}.
  \end{itemize}
\end{theorem}

\begin{proof} Let \begin{math}d\end{math} be the dimension of the vectors. We first prove that the set of vectors
  \begin{equation*}\Big\{\bfu^{t}\prod_{i\in I}
  M_{n_{i}}\mid n_{i}\in \{0,\ldots,q-1\}, I \text{ finite}\Big\}\end{equation*}
is a generating
  system of the whole \begin{math}d\end{math}-dimensional vector space. This is done by contradiction: assume that there is a coordinate
  transformation such that the first \begin{math}d_{0}<d\end{math} unit vectors form a
  basis of the transformed space spanned by  \begin{math}\{\bfu^{t}\prod_{i\in I}
  M_{n_{i}}\mid n_{i}\in \{0,\ldots,q-1\}, I \text{ finite}\}\end{math}. This
  coordinate transform defines a different linear representation of \begin{math}f\end{math}
  with matrices \begin{math}\hat{M_{i}}\end{math} and vectors \begin{math}\hat{\bfu}\end{math} and
  \begin{math}\hat{\bfv}\end{math}. However, only the first \begin{math}d_{0}\end{math} coordinates of any
  vector \begin{math}\bfu^{t}\prod_{i\in I}M_{n_{i}}\end{math} are nonzero. Thus we can reduce
  the dimension of the matrices and vectors from \begin{math}d\end{math} to \begin{math}d_{0}\end{math} to obtain a new
  linear representation of \begin{math}f\end{math}. This contradicts the minimality of the
  original linear representation. 

 Analogously,
  \begin{math}\{\prod_{j\in J} M_{n_{j}}\bfv\mid n_{j}\in \{0,\ldots,q-1\}, J
  \text{ finite}\}\end{math}  is also a generating system for the whole vector space.

  The \begin{math}q\end{math}-quasimultiplicativity of \begin{math}f(n)\end{math} with parameter \begin{math}r\end{math} is equivalent to the identity
  \begin{equation*}
    \bfu^{t}\prod_{i\in I} M_{n_{i}}(M_{0}^{r}-\bfv\bfu^{t})\prod_{j\in J}M_{n_{j}}\bfv=0
  \end{equation*}
  for all finite tuples \begin{math}(n_{i})_{i\in I}\end{math} and \begin{math}(n_{j})_{j\in
    J}\end{math}. Since both \begin{math}\{\bfu^{t}\prod_{i\in I} M_{n_{i}} \}\end{math} and \begin{math}\{\prod_{j\in J}M_{n_{j}}\bfv\}\end{math} are generating systems of the entire vector space, this is equivalent to
  \begin{math}\boldsymbol{x}^{t}(M_{0}^{r}-\bfv\bfu^{t})\boldsymbol{y}=0\end{math} for all vectors \begin{math}\boldsymbol{x}\end{math} and \begin{math}\boldsymbol{y}\end{math}, which in turn is
  equivalent to \begin{math}M_{0}^{r}=\bfv\bfu^{t}\end{math}.
\end{proof}

\begin{example}[The number of optimal \begin{math}\{0,1,-1\}\end{math}-representations]
  The number of optimal \begin{math}\{0,1,-1\}\end{math}-repre\-sentations as described in
  Section~\ref{sec:exampl-q-quasiadd} is a \begin{math}2\end{math}-regular sequence by
  Lemma~\ref{lemma:opt-representations-recursion}. A
  minimal zero-insensitive linear representation for the vector \begin{math}(u_{1}(n),
  u_{2}(n), u_{3}(n), u_{1}(n+1), u_{4}(n+1),
  u_{5}(n+1))^{t}\end{math} is given by
  \begin{equation*}
    M_{0}=
    \begin{pmatrix}
      1&0&0&0&0&0\\
      1&0&0&0&0&0\\
      0&1&0&0&0&0\\
      0&1&0&0&1&0\\
      0&0&0&0&0&1\\
      0&0&0&0&0&0
    \end{pmatrix},\quad
    M_{1}=
    \begin{pmatrix}
      0&1&0&0&1&0\\
      0&0&1&0&0&0\\
      0&0&0&0&0&0\\
      0&0&0&1&0&0\\
      0&0&0&1&0&0\\
      0&0&0&0&1&0
    \end{pmatrix},
  \end{equation*}
\begin{math}\bfu^{t}=(1,0,0,0,0,0)\end{math} and \begin{math}\bfv=(1,1,1,1,0,0)^{t}\end{math}.

As \begin{math}M_{0}^{3}=\bfv\bfu^{t}\end{math}, this sequence is \begin{math}2\end{math}-quasimultiplicative with
parameter \begin{math}3\end{math}, which is the same result as in Lemma~\ref{lem:optrep}. 
\end{example}

\begin{remark}
  The condition on the minimality of the linear representation in
  Theorem~\ref{theorem:reg-mult} is necessary as illustrated by the
  following example:
  
  Consider the sequence \begin{math}f(n)=2^{s_{2}(n)}\end{math}, where \begin{math}s_{2}(n)\end{math} is the binary sum of digits
  function. This sequence is \begin{math}2\end{math}-regular and
  \begin{math}2\end{math}-(quasi-)multiplicative with parameter \begin{math}r=0\end{math}. A ($1$-dimensional) minimal
 linear representation is given by \begin{math}M_{0}=1\end{math}, \begin{math}M_{1}=2\end{math}, \begin{math}v=1\end{math} and \begin{math}u=1\end{math}. As stated
  in Theorem~\ref{theorem:reg-mult}, we have \begin{math}M_{0}^{0}=vu^{t}=1\end{math}.

  If we use the zero-insensitive non-minimal linear representation defined by \begin{math}M_{0}=\big(
  \begin{smallmatrix}
    1&13\\0&2
  \end{smallmatrix}\big)
\end{math}, \begin{math}M_{1}=\big(
\begin{smallmatrix}
  2&27\\0&5
\end{smallmatrix}
\big)\end{math}, \begin{math}\bfv=(1, 0)^{t}\end{math} and \begin{math}\bfu^{t}=(1, 0)\end{math} instead, we have \begin{math}\rank M_{0}^{r}=2\end{math}
for all \begin{math}r\geq 0\end{math}. Thus \begin{math}M_{0}^{r}\neq \bfv\bfu^{t}\end{math}.
\end{remark}

\subsection{When is a $q$-regular function $q$-quasiadditive?}

The characterisation of \begin{math}q\end{math}-regular functions that are also
\begin{math}q\end{math}-quasiadditive is somewhat more complicated. Again, we consider a
zero-insensitive (but not necessarily minimal) linear representation. We let \begin{math}U\end{math} be the smallest
vector space such that all vectors of the form \begin{math}\bfu^{t}\prod_{i\in I}
M_{n_{i}}\end{math} lie in the affine subspace \begin{math}\bfu^{t} + U^t\end{math} (\begin{math}U^t\end{math} is used
as a shorthand for \begin{math}\{\bfx^{t} \,:\, \bfx \in U\}\end{math}). Such a vector
space must exist, since \begin{math}\bfu^{t}\end{math} is a vector of this form
(corresponding to the empty product, where \begin{math}I = \emptyset\end{math}). Likewise,
let \begin{math}V\end{math} be the smallest vector space such that all vectors of the form
\begin{math}\prod_{j\in J}M_{n_{j}}\bfv\end{math} lie in the affine subspace \begin{math}\bfv +
V\end{math}.

\begin{theorem}\label{thm:q-reg-q-quasiadd}
  Let \begin{math}f\end{math} be a \begin{math}q\end{math}-regular sequence with zero-insensitive linear representation
  \eqref{eq:q-regular}. The sequence \begin{math}f\end{math} is \begin{math}q\end{math}-quasiadditive with parameter \begin{math}r\end{math} if and only if all of the following statements hold:
\begin{itemize}
\item \begin{math}\bfu^t \bfv = 0\end{math},
\item \begin{math}U^t\end{math} is orthogonal to \begin{math}(M_0^r - I)\bfv\end{math}, i.e., \begin{math}\bfx^t(M_0^r - I)\bfv = \bfx^tM_0^r\bfv - \bfx^t\bfv = 0\end{math} for all \begin{math}\bfx \in U\end{math},
\item \begin{math}V\end{math} is orthogonal to \begin{math}\bfu^t(M_0^r - I)\end{math}, i.e., \begin{math}\bfu^t(M_0^r - I)\bfy = \bfu^tM_0^r\bfy - \bfu^t\bfy = 0\end{math} for all \begin{math}\bfy \in V\end{math},
\item \begin{math}U^t M_0^r V = 0\end{math}, i.e., \begin{math}\bfx^t M_0^r \bfy  = 0\end{math} for all \begin{math}\bfx \in U\end{math} and \begin{math}\bfy \in V\end{math}.
\end{itemize}
\end{theorem}

\begin{proof}
The first statement \begin{math}\bfu^t \bfv =0\end{math} is equivalent to \begin{math}f(0)=0\end{math}, which we already know to be a necessary condition by Lemma~\ref{lem:simplefacts}. Note also that \begin{math}\bfu^t M_0^r \bfv = \bfu^t \bfv = 0\end{math} by the assumption that the linear representation is zero-insensitive.
For the remaining statements, we write the quasiadditivity condition in terms of our matrix representation as we did in the quasimultiplicative case:
  \begin{equation*}
\bfu^{t}\prod_{i\in I} M_{n_{i}} M_{0}^{r} \prod_{j\in J}M_{n_{j}} \bfv =    \bfu^{t}\prod_{i\in I} M_{n_{i}} \bfv +    \bfu^{t} \prod_{j\in J}M_{n_{j}} \bfv.
  \end{equation*}
Specifically, when \begin{math}J = \emptyset\end{math}, we get
\begin{equation*}
\bfu^{t}\prod_{i\in I} M_{n_{i}} \big( M_{0}^{r} - I \big) \bfv = \bfu^t \bfv = 0.\end{equation*}
Setting also \begin{math}I = \emptyset\end{math} gives us \begin{math}\bfu^{t} (M_0^r-I) \bfv = 0\end{math},
so together we obtain
\begin{equation*}\big( \bfu^{t}\prod_{i\in I} M_{n_{i}} - \bfu^t \big) \big( M_{0}^{r} - I \big) \bfv = 0.\end{equation*}
Since \begin{math}U^t\end{math} is spanned by all vectors of the form \begin{math}\bfu^{t}\prod_{i\in
  I} M_{n_{i}} - \bfu^t\end{math}, the second statement follows. The proof of the third statement is analogous. Finally, if we assume that the first three statements hold, then we find that
\begin{align*}
&\bfu^{t}\prod_{i\in I} M_{n_{i}} M_{0}^{r} \prod_{j\in J}M_{n_{j}} \bfv \\
&=
\big( \bfu^{t}\prod_{i\in I} M_{n_{i}} - \bfu^t \big) M_{0}^{r} \big( \prod_{j\in J}M_{n_{j}} \bfv - \bfv \big) +
\big( \bfu^{t}\prod_{i\in I} M_{n_{i}} - \bfu^t \big) M_{0}^{r} \bfv + 
\bfu^t M_{0}^{r} \big( \prod_{j\in J}M_{n_{j}} \bfv - \bfv \big)\\&\quad +
\bfu^t M_{0}^{r} \bfv \\
&= \big( \bfu^{t}\prod_{i\in I} M_{n_{i}} - \bfu^t \big) M_{0}^{r} \big( \prod_{j\in J}M_{n_{j}} \bfv - \bfv \big) +
\big( \bfu^{t}\prod_{i\in I} M_{n_{i}} - \bfu^t \big) \bfv + 
\bfu^t \big( \prod_{j\in J}M_{n_{j}} \bfv - \bfv \big) \\
&= \big( \bfu^{t}\prod_{i\in I} M_{n_{i}} - \bfu^t \big) M_{0}^{r} \big( \prod_{j\in J}M_{n_{j}} \bfv - \bfv \big)
+ \bfu^{t}\prod_{i\in I} M_{n_{i}} \bfv + \bfu^{t}\prod_{j\in J} M_{n_{j}} \bfv.
\end{align*}
Thus \begin{math}q\end{math}-quasiadditivity is equivalent to
\begin{equation*}\big( \bfu^{t}\prod_{i\in I} M_{n_{i}} - \bfu^t \big) M_{0}^{r} \big( \prod_{j\in J}M_{n_{j}} \bfv - \bfv \big) = 0\end{equation*}
being valid for all choices of \begin{math}I, J\end{math}, \begin{math}n_i\end{math} and \begin{math}n_j\end{math}. The desired fourth condition is clearly equivalent by definition of \begin{math}U\end{math} and \begin{math}V\end{math}.
\end{proof}

\begin{example}
  For the Hamming weight of the nonadjacent form, a zero-insensitive (and also minimal) linear representation for the vector \begin{math}(\hn(n),\hn(n+1),\hn(2n+1),1)^{t}\end{math} is
  \begin{equation*}
    M_{0}=
    \begin{pmatrix}
      1&0&0&0\\0&0&1&0\\1&0&0&1\\0&0&0&1
    \end{pmatrix},\quad
    M_{1}=
    \begin{pmatrix}
      0&0&1&0\\0&1&0&0\\0&1&0&1\\0&0&0&1
    \end{pmatrix},
  \end{equation*}
\begin{math}\bfu^{t}=(1,0,0,0)\end{math} and \begin{math}\bfv=(0,1,1,1)^{t}\end{math}.

The three vectors \begin{math}\mathbf{w}_1 = \bfu^{t}M_{1}-\bfu^{t}\end{math},
\begin{math}\mathbf{w}_2 = \bfu^{t}M_{1}^{2}-\bfu^{t}\end{math} and
\begin{math}\mathbf{w}_3 =  \bfu^{t}M_{1}M_{0}M_{1}-\bfu^{t}\end{math} are linearly
independent. If we let \begin{math}W\end{math} be the vector space spanned by those three, it is easily verified that \begin{math}M_{0}\end{math} and \begin{math}M_{1}\end{math} 
map the affine subspace \begin{math}\bfu^{t}+ W^t\end{math} to itself, so \begin{math}U=W\end{math} is spanned by these vectors.

Similarly, the three vectors \begin{math}M_{1}\bfv-\bfv\end{math},
\begin{math}M_{1}^{2}\bfv-\bfv\end{math} and
\begin{math}M_{1}M_{0}M_{1}\bfv-\bfv\end{math} span \begin{math}V\end{math}.

The first condition of Theorem~\ref{thm:q-reg-q-quasiadd} is obviously
true. We only have to verify the other three conditions with \begin{math}r=2\end{math} for the base vectors
of \begin{math}U\end{math} and \begin{math}V\end{math}, which is done easily. Thus \begin{math}\hn\end{math} is a \begin{math}2\end{math}-regular
sequence that is also \begin{math}2\end{math}-quasiadditive, as was also proved in Section~\ref{sec:exampl-q-quasiadd}.
\end{example}

Finding the vector spaces \begin{math}U\end{math}
and \begin{math}V\end{math} is not trivial. But in a certain special
case of \begin{math}q\end{math}-regular functions, we can give a sufficient condition for
\begin{math}q\end{math}-additivity, which is easier to check. These \begin{math}q\end{math}-regular functions are output sums of
transducers as defined
in~\cite{Heuberger-Kropf-Prodinger:2015:output}: a transducer
transforms the \begin{math}q\end{math}-ary expansion of an integer \begin{math}n\end{math}  deterministically into an output
sequence. We are
interested in the sum of this output sequence. Before we can state our
condition, we introduce our notation more precisely.

A transducer consists of a finite number of states, an
initial state, the
input alphabet $\{0,\ldots,q-1\}$, an output alphabet, which is a
subset of the real numbers, and transitions between two states with labels
$\eps\mid\delta$ for $\eps$ an input letter and $\delta$ an output
letter.  We assume that the transducer is complete and deterministic,
that is for every state $s$ and input letter $\eps$, there exists
exactly one transition leaving state $s$ with input label
$\eps$. Additionally every state has a final output. 

The transducer reads the $q$-ary expansion of an integer $n$, starting from the
least significant digit, as input, which defines a unique path starting at the initial state with the given
input as input label. The output of the transducer is the sequence of
output labels along this path together with the final output of the
final state of this path. The output sum is then the sum of this
output sequence.

The function \begin{math}\hn\end{math}, see
Example~\ref{ex:hnaf-trans}, as well as many other examples, can be
represented in this way.

This output sum of a transducer is a $q$-regular sequence
\cite{Heuberger-Kropf-Prodinger:2015:output}. To obtain a linear representation, we define  the matrix $N_{\eps}$ for $\eps\in\{0,\ldots,q-1\}$
to be the adjacency matrix of the transducer where we only take into
account transitions with input label $\eps$. Note that because our
transducer is complete and deterministic, there is exactly one entry
$1$ in every row. Without loss of
generality, we say that the initial state corresponds to the first row
and column. Furthermore, the $i$-th entry of the vector
$\bfdelta_{\eps}$ is the output label of the transition starting in
state $i$ with input label $\eps$. We define the matrices
\begin{equation*}
  M_{\eps}=\begin{pmatrix}
    N_{\eps}&\bfdelta_{\eps}&[\eps=0]I\\
    \boldsymbol{0}&1&\boldsymbol{0}\\
    0&0&[\eps=0]I
  \end{pmatrix},
\end{equation*}
where $I$ is an identity matrix of the correct size, and we set \begin{math}\bfu^{t}=(1,0,\ldots,0)\end{math} and
\begin{equation*}
\bfv=
\begin{pmatrix}
  \bfb(0)\\1\\\bfb(0)-N_{0}\bfb(0)-\bfdelta_{0}
\end{pmatrix},
\end{equation*}
where the entries of $\bfb(0)$ are the final outputs of the states.

Following \cite[Remark~3.10]{Heuberger-Kropf-Prodinger:2015:output},
the output sum of a transducer is $q$-regular with the linear
representation $(\bfu, (M_{\eps})_{0\leq\eps<q}, \bfv)$.

\begin{example}\label{ex:hnaf-trans}
\begin{figure}
\newcommand{\Bold}[1]{\mathbf{#1}}
\begin{tikzpicture}[auto, initial text=, >=latex, accepting text=, accepting/.style=accepting by arrow, accepting distance=5ex, every state/.style={minimum
    size=1.3em}]
\node[state, initial] (v0) at (0.000000,
0.000000) {};
\path[->] (v0.270.00) edge node[rotate=450.00, anchor=south] {$0$} ++(270.00:5ex);
\node[state] (v1) at (3.000000, 0.000000) {};
\path[->] (v1.270.00) edge node[rotate=450.00, anchor=south] {$0$} ++(270.00:5ex);
\node[state] (v2) at (6.000000, 0.000000) {};
\path[->] (v2.270.00) edge node[rotate=450.00, anchor=south] {$1$} ++(270.00:5ex);
\path[->] (v0.10.00) edge node[rotate=0.00, anchor=south] {$1\mid 1$} (v1.170.00);
\path[->] (v1.10.00) edge node[rotate=0.00, anchor=south] {$1\mid 0$} (v2.170.00);
\path[->] (v0) edge[loop above] node {$0\mid 0$} ();
\path[->] (v2.190.00) edge node[rotate=360.00, anchor=north] {$0\mid 1$} (v1.350.00);
\path[->] (v2) edge[loop above] node {$1\mid 0$} ();
\path[->] (v1.190.00) edge node[rotate=360.00, anchor=north] {$0\mid 0$} (v0.350.00);
\end{tikzpicture}
\caption{Transducer to compute the Hamming weight of the nonadjacent form.}
\label{fig:NAF}
\end{figure}
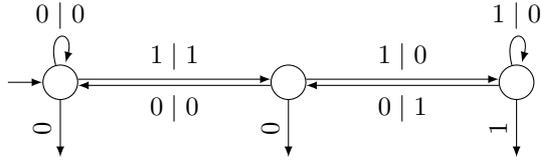
  The output sum of the transducer in Figure~\ref{fig:NAF} is exactly
  the Hamming weight of the nonadjacent form $\hn(n)$ (see, e.g., \cite{Heuberger-Kropf-Prodinger:2015:output}). The matrices
  and vectors corresponding to this transducer are
  \begin{align*}
    N_{0}=
    \begin{pmatrix}
      1&0&0\\
      1&0&0\\
      0&1&0
    \end{pmatrix},\qquad N_{1}=
    \begin{pmatrix}
      0&1&0\\
      0&0&1\\
      0&0&1
    \end{pmatrix},
  \end{align*}
 $\bfdelta_{0}^{t}=(0,0,1)$, $\bfdelta_{1}^{t}=(1,0,0)$ and $\bfb(0)^{t}=(0,0,1)$.
\end{example}

To state our condition, we also introduce the notion of a reset
sequence: a reset sequence is an input sequence which always leads to
the same state no matter in which state of the transducer we start.
Not every transducer has a reset sequence, not even every strongly
connected transducer has one. In many cases arising from combinatorics
and digit expansions
the reset sequence consists only of zeros.
\begin{prop}\label{proposition:q-add-transducer}
  The output sum of a connected transducer is \begin{math}q\end{math}-additive with parameter \begin{math}r\end{math} if the following
  conditions are satisfied:
  \begin{itemize}
  \item The transducer has the reset sequence \begin{math}0^{r}\end{math} leading to the
    initial state.
  \item For every state, the output sum along the path of the reset
    sequence \begin{math}0^{r}\end{math} equals the final output of
    this state.
  \item Additional zeros at the end of the input sequence do not
    change the output sum.
  \end{itemize}
\end{prop}

\begin{proof}
Let $\bff(n)$ be the vector corresponding to the linear representation
$(\bfu, (M_{\eps})_{0\leq\eps<q}, \bfv)$ as defined
in~\eqref{eq:q-regular-recursive}. By induction, we obtain that the
middle coordinate of $\bff(n)$ is always $1$ and the coordinates below
are always $0$ if $n\geq1$. We denote the coordinates above by $\bfb(n)$.
The output sum of the transducer is
the first coordinate of $\bfb(n)$. By~\eqref{eq:q-regular-recursive},
we obtain the recursion
\begin{equation}\label{eq:output-sum-recursive}
  \bfb(qn+\eps)=N_{\eps}\bfb(n)+\bfdelta_{\eps}
\end{equation}
if \begin{math}qn+\eps>0\end{math}.

 The third condition ensures that leading zeros does not
 change anything. Thus the connectivity of the underlying graph
 implies that \eqref{eq:output-sum-recursive} also holds
for \begin{math}qn+\eps=0\end{math}. Thus, the last
coordinates of \begin{math}\bfv\end{math} are zero and we could reduce the dimension of the
linear representation.

Let \begin{math}J\end{math} be finite and \begin{math}n_{j}\in\{0,\ldots,q-1\}\end{math} for \begin{math}j\in J\end{math}. The first condition implies that
\begin{equation*}
  \prod_{j\in J}N_{n_{j}}N_{0}^{r}=
  \begin{pmatrix}
    1&0&\cdots&0\\
    \vdots&\vdots&\ddots&\vdots\\
    1&0&\cdots&0
  \end{pmatrix},
\end{equation*}
and the second condition implies that
\begin{equation*}
  \prod_{j\in J}N_{n_{j}}\bfb(0)=\prod_{j\in J}N_{n_{j}}(I+\cdots+N_{0}^{r-1})\bfdelta_{0}.
\end{equation*}

Using \eqref{eq:output-sum-recursive} recursively together with these two conditions gives
\begin{align*}
  \bfb(q^{k+r}m+n)&=\prod_{j=0}^{k-1}N_{n_{j}}\bfb(q^{r}m)+\sum_{j=0}^{k-1}\prod_{i=0}^{j-1}N_{n_{i}}\bfdelta_{n_{j}}\\
&=\prod_{j=0}^{k-1}N_{n_{j}}N_{0}^{r}\bfb(m)+\prod_{j=0}^{k-1}N_{n_{j}}(I+\cdots+N_{0}^{r-1})\bfdelta_{0}+\bfb(n)-\prod_{j=0}^{k-1}N_{n_{j}}\bfb(0)\\
&=\begin{pmatrix}
    1&0&\cdots&0\\
    \vdots&\vdots&\ddots&\vdots\\
    1&0&\cdots&0
  \end{pmatrix}\bfb(m)+\bfb(n)
\end{align*}
for all \begin{math}n\end{math} with $q$-ary digit expansion \begin{math}(n_{k-1}\cdots n_{0})\end{math} and all
\begin{math}m\end{math}. This implies that the first coordinate of \begin{math}\bfb(n)\end{math} is \begin{math}q\end{math}-quasiadditive.
\end{proof}

\begin{example}
  We now continue Example~\ref{ex:hnaf-trans} and check whether the
  conditions of Proposition~\ref{proposition:q-add-transducer} are
  satisfied for the transducer given in Figure~\ref{fig:NAF}. First, a
  reset sequence is $00$ (i.e., $r=2$) and leads to the initial
  state. Second, the output sum  along the path of the reset sequence
  is $0$, $0$ and $1$ for the left, the middle and the right state,
  respectively, which is exactly the final output of the corresponding
  state. Furthermore, leading zeros do not change the output sum. Thus we have another proof that $\hn(n)$ is a 2-quasiadditive function with parameter
  $r=2$.
\end{example}

\section{A central limit theorem for $q$-quasiadditive and -multiplicative functions}
In this section, we prove a central limit theorem for
\begin{math}q\end{math}-quasimultiplicative functions taking only positive values.
By Proposition~\ref{prop:trivial}, this also implies a central
limit theorem for \begin{math}q\end{math}-quasiadditive functions.

To this end, we define a generating function: let \begin{math}f\end{math} be a \begin{math}q\end{math}-quasimultiplicative function with positive values, let \begin{math}\M_k\end{math} be the set of all nonnegative integers less than \begin{math}q^k\end{math} (i.e., those positive integers whose \begin{math}q\end{math}-ary expansion needs at most \begin{math}k\end{math} digits), and set
\begin{equation*}F(x,t) = \sum_{k \geq 0} x^k \sum_{n \in \M_k} f(n)^t.\end{equation*}
The decomposition of Proposition~\ref{prop:split} now translates
directly to an alternative representation for \begin{math}F(x,t)\end{math}: let \begin{math}\B\end{math} be
the set of all positive integers not divisible by \begin{math}q\end{math} whose \begin{math}q\end{math}-ary representation does not contain the block \begin{math}0^{r}\end{math}, let \begin{math}\ell(n)\end{math} denote the length of the \begin{math}q\end{math}-ary representation of \begin{math}n\end{math}, and define the function \begin{math}B(x,t)\end{math} by
\begin{equation*}B(x,t) = \sum_{n \in \B} x^{\ell(n)} f(n)^t.\end{equation*}
We remark that in the special case where \begin{math}q=2\end{math} and \begin{math}r=1\end{math}, this simplifies greatly to
\begin{equation}\label{eq:q2_r1}
B(x,t) = \sum_{k \geq 1} x^{k} f(2^k-1)^t.
\end{equation}

\begin{prop}\label{prop:gf}
The generating function \begin{math}F(x,t)\end{math} can be expressed as
\begin{equation*}F(x,t) = \frac{1}{1-x} \cdot \frac{1}{1 - \frac{x^r}{1-x} B(x,t)} \Big( 1 + (1+x+\cdots+x^{r-1})B(x,t) \Big) = \frac{1+(1+x+\cdots+x^{r-1})B(x,t)}{1-x-x^rB(x,t)}.\end{equation*}
\end{prop}

\begin{proof}
The first factor stands for the initial sequence of leading zeros, the
second factor for a (possibly empty) sequence of blocks consisting of
an element of \begin{math}\B\end{math} and \begin{math}r\end{math} or more zeros, and the last factor for the
final part, which may be empty or an element of \begin{math}\B\end{math} with up to \begin{math}r-1\end{math} zeros (possibly none) added at the end.
\end{proof}

Under suitable assumptions on the growth of a \begin{math}q\end{math}-quasiadditive or \begin{math}q\end{math}-quasimultiplicative function, we can exploit the expression of Proposition~\ref{prop:gf} to prove a central limit theorem.

\begin{defi}
  We say that a function \begin{math}f\end{math} has \emph{at most polynomial growth} if
  \begin{math}f(n)=O(n^{c})\end{math} and \begin{math}f(n) = \Omega(n^{-c})\end{math} for a fixed \begin{math}c\geq
  0\end{math}. We say that \begin{math}f\end{math} has \emph{at most logarithmic growth} if
  \begin{math}f(n)=O(\log n)\end{math}.
\end{defi}

Note that our definition of at most polynomial growth is slightly
different than usual: the extra condition \begin{math}f(n) =
  \Omega(n^{-c})\end{math} ensures that the absolute value
of \begin{math}\log f(n)\end{math} does not grow too fast.

\begin{lemma}\label{lemma:singularity} Assume that the positive, \begin{math}q\end{math}-quasimultiplicative function \begin{math}f\end{math} has at most polynomial growth. 

There exist positive constants \begin{math}\delta\end{math} and \begin{math}\epsilon\end{math} such that
\begin{itemize}
\item \begin{math}B(x,t)\end{math} has radius of convergence \begin{math}\rho(t) > \frac1q\end{math} whenever \begin{math}|t| \leq \delta\end{math}.
\item For \begin{math}|t| \leq \delta\end{math}, the equation \begin{math}x + x^r B(x,t) = 1\end{math} has a complex solution \begin{math}\alpha(t)\end{math} with \begin{math}|\alpha(t)| < \rho(t)\end{math} and no other solutions with modulus \begin{math}\leq (1+\epsilon)|\alpha(t)|\end{math}. 
\item Thus the generating function \begin{math}F(x,t)\end{math} has a simple pole at \begin{math}\alpha(t)\end{math} and no further singularities of modulus \begin{math}\leq (1+ \epsilon)|\alpha(t)|\end{math}. 
\item Finally, \begin{math}\alpha\end{math} is an analytic function of \begin{math}t\end{math} for \begin{math}|t| \leq \delta\end{math}.
\end{itemize}
\end{lemma}

\begin{proof}
The polynomial growth of \begin{math}f\end{math} implies that \begin{math}C^{-1}\phi^{-\ell(n)} \leq f(n)\leq C
\phi^{\ell(n)}\end{math} for some positive constants \begin{math}C\end{math} and \begin{math}\phi\end{math}.
Moreover, \begin{math}\B\end{math} contains \begin{math}O(\beta^{\ell})\end{math} elements whose \begin{math}q\end{math}-ary
expansion has length at most \begin{math}\ell\end{math}, where \begin{math}\beta < q\end{math} is a root of
the polynomial \begin{math}x^r-(q-1)x^{r-1}-\cdots-(q-1)x-(q-1)\end{math}. This implies
that \begin{math}B(x,t)\end{math} is indeed an analytic function of \begin{math}x\end{math} for \begin{math}|x| < \beta^{-1}
\phi^{\delta}\end{math} whenever \begin{math}|t| \leq \delta\end{math}. For suitably small
\begin{math}\delta\end{math}, \begin{math}\beta^{-1} \phi^{\delta}\end{math} is greater than \begin{math}\frac1q\end{math}, which proves the first part of
our statement. Next note that
\begin{equation*}
B(x,0)=\frac{(q-1)x}{1-(q-1)x-\cdots-(q-1)x^{r}},
\end{equation*}
and it follows by an easy calculation that
$$1 - x - x^r B(x,0) = \frac{(1-x)(1-qx)}{1 - q x + (q-1)x^{r+1}}.$$
Hence \begin{math}\alpha(0) = \frac1q\end{math} is the only solution of the equation \begin{math}x + x^r B(x,0) = 1\end{math}, and it is a simple root.
All remaining statements are therefore simple consequences of the implicit function theorem.
\end{proof}

\begin{lemma}\label{lem:sing_anal}
Assume that the positive, \begin{math}q\end{math}-quasimultiplicative function \begin{math}f\end{math} has at most polynomial growth.

With \begin{math}\delta\end{math} and \begin{math}\epsilon\end{math} as in the previous lemma, we have, uniformly in \begin{math}t\end{math},
\begin{equation*}[x^k] F(x,t) = \kappa(t) \cdot \alpha(t)^{-k} \big(1 + O((1+\epsilon)^{-k})\big)\end{equation*}
for some function \begin{math}\kappa\end{math}. Both \begin{math}\alpha\end{math} and \begin{math}\kappa\end{math} are analytic functions of \begin{math}t\end{math} for \begin{math}|t| \leq \delta\end{math}, and \begin{math}\kappa(t) \neq 0\end{math} in this region.
\end{lemma}

\begin{proof}
This follows from the previous lemma by means of singularity analysis,
see \cite[Chapter VI]{Flajolet-Sedgewick:ta:analy}.
\end{proof}

\begin{theorem}\label{thm:clt-mult}
Assume that the positive, \begin{math}q\end{math}-quasimultiplicative function \begin{math}f\end{math} has at most polynomial growth.

Let \begin{math}N_k\end{math} be a randomly chosen integer in \begin{math}\{0,1,\ldots,q^k-1\}\end{math}. The
random variable \begin{math}L_k = \log f(N_k)\end{math} has mean \begin{math}\mu k + O(1)\end{math} and
variance \begin{math}\sigma^2 k + O(1)\end{math}, where the two constants are given by
\begin{equation*}\mu = \frac{B_t(1/q,0)}{q^{2r}}\end{equation*}
and
\begin{multline}
  \sigma^2= -B_{t}(1/q,0)^{2} {q}^{-4r+1}(q-1)^{-1} + 2B_{t}(1/q,0)^{2} {q}^{-3r+1}(q-1)^{-1} -B_{t}(1/q,0)^{2}{q}^{-4r}(q-1)^{-1} \\-
     4rB_{t}(1/q,0)^{2} {q}^{-4r} + B_{tt}(1/q,0){q}^{-2r}
     - 2B_{t}(1/q,0)
        B_{tx}(1/q,0)
{q}^{-4r-1} .
\end{multline}
If \begin{math}f\end{math} is not the constant function \begin{math}f \equiv 1\end{math}, then \begin{math}\sigma^2 \neq 0\end{math} and the normalised random variable \begin{math}(L_k - \mu k)/(\sigma \sqrt{k})\end{math} converges weakly to a standard Gaussian distribution.
\end{theorem}

\begin{proof}
The moment
generating function of \begin{math}L_{k}\end{math} is \begin{math}[x^{k}]F(x,t)/q^{k}\end{math}.
Hence the statement follows from Lemma~\ref{lem:sing_anal} by means of the Quasi-power theorem, see
\cite{Hwang:1998} or \cite[Chapter IX.5]{Flajolet-Sedgewick:ta:analy}. The only part that we actually have to verify is
that \begin{math}\sigma^2 \neq 0\end{math} unless \begin{math}f\end{math} is constant. 

Assume that \begin{math}\sigma^{2}=0\end{math}. We first consider the case that
\begin{math}\log\alpha(t)\end{math} is not a linear function. Let \begin{math}s\end{math} be the least integer
greater than \begin{math}1\end{math} such that \begin{math}t^s\end{math} occurs with a nonzero coefficient in the Taylor expansion of \begin{math}\log \alpha(t)\end{math} at \begin{math}t=0\end{math}, i.e.,
\begin{equation*}\log \alpha(t) = \log \alpha(0) + a t + b t^s + O(t^{s+1}).\end{equation*}
Note that $a = -\mu$. Moreover, by the assumption that $\sigma^2 = 0$, we must have \begin{math}s\geq 3\end{math}. Since \begin{math}\alpha(0) = \frac{1}{q}\end{math} and \begin{math}\kappa(0)=1\end{math}, it follows that
\begin{align*}
\E(\exp(tL_k)) &= \frac{[x^k] F(x,t)}{q^k} = \exp \Big( \log \kappa(t) - k\log \alpha(t) - k \log q + O\big((1+\epsilon)^{-k} \big) \Big) \\
&= \exp\Big( -akt -bkt^s +O \big(kt^{s+1}+t+(1+\epsilon)^{-k} \big) \Big).
\end{align*}
Considering the normalised version \begin{math}R_k = \frac{L_k - \mu k}{k^{1/s}}\end{math} of the random variable \begin{math}L_k\end{math}, we get
\begin{equation*}\E \Big( \exp\big( \tau R_k \big) \Big) = \exp \Big( -b\tau^s + O \big(k^{-1/s} + (1+\epsilon)^{-k} \big)\Big)\end{equation*}
for fixed \begin{math}\tau\end{math}.
So for every complex $\tau$, we have  \begin{math}\lim_{k \to \infty} \E ( \exp\big( \tau R_k \big) ) =
\exp(-b\tau^s)\end{math}, which is a continuous
function. By L\'evy's continuity theorem, this would imply convergence
in distribution of \begin{math}R_k\end{math} to a random variable with moment generating
function \begin{math}M(\tau) = \exp(-b\tau^s)\end{math}. However, there is no such random
variable: all derivatives at \begin{math}\tau=0\end{math} are finite and the second
derivative of \begin{math}\exp(-b\tau^s)\end{math} at \begin{math}\tau=0\end{math} is 0, thus the second
moment is \begin{math}0\end{math}. A random variable whose second moment is \begin{math}0\end{math} is almost surely
equal to \begin{math}0\end{math} and would thus have moment generating function \begin{math}1\end{math}.

The only remaining possibility is that \begin{math}\log \alpha(t)\end{math} is linear:
\begin{math}\log \alpha(t) = \log \alpha(0) + a t\end{math}, thus \begin{math}\alpha(t) = \alpha(0)
e^{at} = e^{at}/q\end{math}. If we plug this into the defining equation of \begin{math}\alpha(t)\end{math}, we obtain
\begin{equation*}1 = \frac{e^{at}}{q} + \frac{e^{art}}{q^r} \sum_{n \in \B} q^{-\ell(n)} e^{a \ell(n)t} f(n)^t\end{equation*}
identically for \begin{math}|t| \leq \delta\end{math}. However, the right side of this identity has strictly positive second derivative for real \begin{math}t\end{math} unless \begin{math}a = 0\end{math} and \begin{math}f(n) = 1\end{math} for all \begin{math}n \in \B\end{math} (in which case \begin{math}f(n) = 1\end{math} for all \begin{math}n\end{math}). Thus \begin{math}\sigma^2 \neq 0\end{math} unless \begin{math}f \equiv 1\end{math}.
\end{proof}

\begin{cor}\label{cor:clt-add}
  Assume that the \begin{math}q\end{math}-quasiadditive function \begin{math}f\end{math} has at most logarithmic growth.

Let \begin{math}N_k\end{math} be a randomly chosen integer in \begin{math}\{0,1,\ldots,q^k-1\}\end{math}. The
random variable \begin{math}L_k = f(N_k)\end{math} has mean \begin{math}\hat\mu k + O(1)\end{math} and
variance \begin{math}\hat\sigma^2 k + O(1)\end{math}, where the two constants \begin{math}\hat\mu\end{math} and \begin{math}\hat\sigma^2\end{math}are given by
the same formulas as in Theorem~\ref{thm:clt-mult}, with \begin{math}B(x,t)\end{math} replaced by
\begin{equation*}
  \hat B(x,t) = \sum_{n \in \B} x^{\ell(n)} e^{f(n)t}.
\end{equation*}

If \begin{math}f\end{math} is not the constant function \begin{math}f \equiv 0\end{math}, then the normalised random variable \begin{math}(L_k - \hat\mu k)/(\hat\sigma \sqrt{k})\end{math} converges weakly to a standard Gaussian distribution.
\end{cor}

\begin{remark}
By means of the Cram\'er-Wold device (and Corollary~\ref{cor:lin_comb}), we also obtain joint normal distribution of tuples of \begin{math}q\end{math}-quasiadditive functions.
\end{remark}

We now revisit the examples discussed in
Section~\ref{sec:exampl-q-quasiadd} and state the corresponding
central limit theorems. Some of them are well known while others are
new. We also provide numerical values for the constants in mean and variance.
\begin{example}[see also \cite{Kirschenhofer:1983:subbl,Drmota:2000}]The number of blocks \begin{math}0101\end{math} occurring in the binary
  expansion of \begin{math}n\end{math} is a \begin{math}2\end{math}-quasiadditive function of at most
  logarithmic growth. Thus by Corollary~\ref{cor:clt-add}, the
  standardised random variable is asymptotically normally distributed, the constants being \begin{math}\hat\mu = \frac1{16}\end{math} and \begin{math}\hat\sigma^2 = \frac{17}{256}\end{math}.
\end{example}

\begin{example}[see also \cite{Thuswaldner:1999,Heuberger-Kropf:2013:analy}]
  The Hamming weight of the nonadjacent form is \begin{math}2\end{math}-quasiadditive
  with at most logarithmic growth (as the length of the NAF of \begin{math}n\end{math} is logarithmic). Thus by Corollary~\ref{cor:clt-add}, the
  standardised random variable is asymptotically normally distributed. The associated constants are \begin{math}\hat\mu = \frac13\end{math} and \begin{math}\hat\sigma^2 = \frac2{27}\end{math}.
\end{example}

\begin{example}[see Section~\ref{sec:exampl-q-quasiadd}]
  The number of optimal \begin{math}\{0,1,-1\}\end{math}-representations is
  \begin{math}2\end{math}-quasi\-mul\-ti\-plica\-tive. As it is always greater or equal to \begin{math}1\end{math} and
  \begin{math}2\end{math}-regular, it has at most polynomial growth. Thus
  Theorem~\ref{thm:clt-mult} implies that the standardised logarithm
  of this random variable is asymptotically normally distributed with
  numerical constants given by \begin{math}\mu\approx 0.060829\end{math}, \begin{math}\sigma^{2}\approx 0.038212\end{math}.
\end{example}

\begin{example}[see Section~\ref{sec:exampl-q-quasiadd}] Suppose that the sequence \begin{math}s_1,s_2,\ldots\end{math} satisfies \begin{math}s_{n}\geq 1\end{math} and \begin{math}s_{n}=O(c^{n})\end{math} for a constant
  \begin{math}c\geq 1\end{math}.
  The run length transform \begin{math}t(n)\end{math} of \begin{math}s_{n}\end{math}
  is \begin{math}2\end{math}-quasimultiplicative. As \begin{math}s_{n}\geq 1\end{math} for all \begin{math}n\end{math}, we have \begin{math}t(n)\geq
  1\end{math} for all \begin{math}n\end{math} as well. Furthermore, there exists a constant \begin{math}A\end{math} such that \begin{math}s_n \leq A c^n\end{math} for all \begin{math}n\end{math}, and the sum of all run lengths is bounded by the length of the   binary expansion, thus
  \begin{equation*}
t(n)=\prod_{i\in\mathcal L(n)}s_{i} \leq \prod_{i \in \mathcal{L}(n)} (A c^i) \leq (Ac)^{1+\log_2 n}.
\end{equation*}
Consequently, \begin{math}t(n)\end{math} is positive and has at most polynomial growth. By
Theorem~\ref{thm:clt-mult}, we obtain an asymptotic normal
distribution for the standardised random variable \begin{math}\log t(N_{k})\end{math}. The constants \begin{math}\mu\end{math} and \begin{math}\sigma^2\end{math} in mean and variance are given by
\begin{equation*}
\mu = \sum_{i \geq 1} (\log s_i) 2^{-i-2} 
\end{equation*}
and
\begin{equation*}
\sigma^2 = \sum_{i \geq 1} (\log s_i)^2 \big(2^{-i-2} - (2i-1)2^{-2i-4} \big) - \sum_{j > i \geq 1} (\log s_i)(\log s_j)   (i+j-1) 2^{-i-j-3}.
\end{equation*}
These formulas can be derived from those given in Theorem~\ref{thm:clt-mult} by means of the representation~\eqref{eq:q2_r1}, and the terms can also be interpreted easily: write \begin{math}\log t(n) = \sum_{i \geq 1} X_i(n) \log s_i\end{math}, where \begin{math}X_i(n)\end{math} is the number of runs of length \begin{math}i\end{math} in the binary representation of \begin{math}n\end{math}. The coefficients in the two formulas stem from mean, variance and covariances of the \begin{math}X_i(n)\end{math}.

In the special case that
\begin{math}s_{n}\end{math} is the Jacobsthal sequence ($s_n = \frac13(2^{n+2} - (-1)^n$), see Section~\ref{sec:exampl-q-quasiadd}), we have the
  numerical values
  \begin{math}\mu \approx 0.429947\end{math}, \begin{math}\sigma^{2} \approx 0.121137\end{math}.
\end{example}

Let us finally show that the central limit theorem holds in a slightly more general version, where we pick an integer uniformly at random from the set $\{0,1,2,\ldots,K-1\}$ ($K$ not necessarily being a power of $q$ any longer). We first state and prove our result for $q$-quasiadditive functions; it automatically transfers to $q$-quasimultiplicative functions by Proposition~\ref{prop:trivial}.

\begin{theorem}
 Assume that the \begin{math}q\end{math}-quasiadditive function \begin{math}f\end{math} has at most logarithmic growth, and that $f$ is not the constant function $f \equiv 0$. Let \begin{math}M_K\end{math} be a randomly chosen integer in \begin{math}\{0,1,\ldots,K-1\}\end{math}. The random variable 
$$\frac{f(M_K) - \hat\mu \log_q K}{\hat\sigma \sqrt{\log_q K}},$$
where the two constants $\hat\mu$ and $\hat\sigma^2$ are the same as in Corollary~\ref{cor:clt-add}, converges weakly to a standard Gaussian distribution.
\end{theorem}

\begin{proof}
Let $L_1$ and $L_2$ be the largest integers for which we have $q^{L_1}
< K/\log^2 K$ and $q^{L_2} < K/\log K$, respectively. For each
nonnegative integer $m < K$, we consider (if it exists) a representation of the form
\begin{equation}\label{eq:good_rep}
m = q^{k+r} a + b,
\end{equation}
where $b < q^k$ and $L_1 \leq k \leq L_2$. If there are two or more such representations for a specific $m$, we take the one for which $k$ is maximal so as to obtain a unique representation. If $m$ does not have a representation of this form, then it does not have $r$ consecutive zeros in its $q$-ary representation anywhere in the block ranging from the $(L_1+1)$-th to the $(L_2+r)$-th digit, counting from the least significant digit. The proportion of such integers is
$$O \Big( (1-q^{-r})^{(L_2-L_1)/r} \Big) = O \Big( (1-q^{-r})^{(\log\log K)/r} \Big),$$
which becomes negligible as $K \to \infty$. 

If however $m$ can be represented in the form~\eqref{eq:good_rep}, then we have
$$f(m) = f(a) + f(b)$$
by quasiadditivity of $f$. Moreover, $a = O(\log^2 K)$ by the
definition of $L_1$, so $f(a) = O(\log \log K)$ since we assumed $f$
to have at most logarithmic growth. For given $a$ and $k$, $b$ can be
any integer in the set $\{0,1,\ldots,q^k-1\}$, unless $a = \lfloor
K/q^{k+r} \rfloor$. In the former case, we can identify $b$ with $N_k$, the random variable defined in Theorem~\ref{thm:clt-mult} and
Corollary~\ref{cor:clt-add}. The latter case, however, is negligible, since it only accounts for a proportion of at most
$$\frac{1}{K}\sum_{k=L_1}^{L_2} q^k = O(q^{L_2}/K) = O(1/\log K)$$
values of $m$. Now we condition on the event that the random integer $M_K$ has a representation of the form~\eqref{eq:good_rep} for certain fixed $k$ and $a \neq \lfloor
K/q^{k+r} \rfloor$. For every real number $x$, we have
\begin{equation}\label{eq:conv-distr-M}
\begin{aligned}
\Prob \Big( f(M_K) \leq \hat\mu \log_q K + x \hat \sigma \sqrt{\log_q K} \,\Big|\,& q^{k+r} a \leq M_K  < q^{k+r} a + q^k \Big) \\
&= \Prob \Big( f(N_k) \leq \hat \mu \log_q K + x \hat \sigma
\sqrt{\log_q K} - f(a) \Big)\\
&=\Prob\Big(\frac{f(N_{k})-\hat \mu \log_q K - f(a)}{\hat \sigma
\sqrt{\log_q K}}\leq x\Big).
\end{aligned}
\end{equation}
Note that $k = \log_q K + O(\log \log K)$, so
$$\frac{f(N_{k})-\hat \mu \log_q K - f(a)}{\hat \sigma
\sqrt{\log_q K}} = \frac{f(N_k) - \hat\mu k}{\hat\sigma \sqrt{k}} + O \Big( \frac{\log \log K}{\sqrt{\log K}} \Big).$$
Let $\Phi(x) =  \frac{1}{\sqrt{2\pi}} \int_{-\infty}^x e^{-t^2/2}\,dt$ denote the distribution function of a standard Gaussian distribution.
By Corollary~\ref{cor:clt-add}, and because $\Phi$ is continuous, we have 
$$\Prob\Big(\frac{f(N_{k})-\hat \mu \log_q K - f(a)}{\hat \sigma \sqrt{\log_q K}}\leq x\Big) = \Phi(x) + o(1),$$ 
and this holds uniformly in $x$, $a$ and $k$ as $K \to \infty$ (in fact, one can make the speed of convergence explicit by means of the Quasi-power theorem).
Summing~\eqref{eq:conv-distr-M} over all possible values of $a$ and $k$, we obtain
$$\lim_{K \to \infty} \Prob \Big( f(M_K) \leq \hat\mu \log_q K + x \hat \sigma \sqrt{\log_q K} \,\Big) = \Phi(x)$$
for all real numbers $x$, which is what we wanted to prove.
\end{proof}

\begin{cor}
 Assume that the positive, \begin{math}q\end{math}-quasimultiplicative function \begin{math}f\end{math} has at most polynomial growth, and that $f$ is not the constant function $f \equiv 1$. Let \begin{math}M_K\end{math} be a randomly chosen integer in \begin{math}\{0,1,\ldots,K-1\}\end{math}. The random variable 
$$\frac{f(M_K) - \mu \log_q K}{\sigma \sqrt{\log_q K}},$$
where the two constants $\mu$ and $\sigma^2$ are the same as in Theorem~\ref{thm:clt-mult}, converges weakly to a standard Gaussian distribution.
\end{cor}

\bibliographystyle{amsplain}
\bibliography{lit}

\providecommand{\Submitted}{Submitted} \providecommand{\availableat}{ available
  at } \providecommand{\alsoavailableat}{ also available at }
  \providecommand{\evavailableat}{earlier version available at }
  \providecommand{\toappearin}{To appear in } \providecommand{\toappear}{to
  appear} \providecommand{\inpreparation}{in preparation}
  \providecommand{\doi}[1]{\href{http://dx.doi.org/#1}{\path{doi:#1}}}
  \providecommand{\etc}{\emph{etc.}}\def\cprime{$'$}
\providecommand{\bysame}{\leavevmode\hbox to3em{\hrulefill}\thinspace}
\providecommand{\MR}{\relax\ifhmode\unskip\space\fi MR }
\providecommand{\MRhref}[2]{%
  \href{http://www.ams.org/mathscinet-getitem?mr=#1}{#2}
}
\providecommand{\href}[2]{#2}
\begin{thebibliography}{10}

\bibitem{allouche:1993}
Jean-Paul Allouche and Olivier Salon, \emph{Sous-suites polynomiales de
  certaines suites automatiques}, J. Th\'eor. Nombres Bordeaux \textbf{5}
  (1993), no.~1, 111--121.

\bibitem{Allouche-Shallit:2003:autom}
Jean-Paul Allouche and Jeffrey Shallit, \emph{Automatic sequences: Theory,
  applications, generalizations}, Cambridge University Press, Cambridge, 2003.

\bibitem{Bassily-Katai:1995:distr}
Nader~L. Bassily and Imre K{\'a}tai, \emph{Distribution of the values of
  {$q$}-additive functions on polynomial sequences}, Acta Math. Hungar.
  \textbf{68} (1995), no.~4, 353--361.

\bibitem{Bellman-Shapiro:1948}
Richard Bellman and Harold~N. Shapiro, \emph{On a problem in additive number
  theory}, Ann. of Math. (2) \textbf{49} (1948), no.~2, 333--340.

\bibitem{Cateland:digital-seq}
Emmanuel Cateland, \emph{Suites digitales et suites k-r\'{e}guli\`{e}res},
  1992, Doctoral thesis, Universit\'{e} de Bordeaux.

\bibitem{Delange:1972:q-add-q-mult}
Hubert Delange, \emph{Sur les fonctions $q$-additives ou $q$-multiplicatives},
  Acta Arith. \textbf{21} (1972), 285--298.

\bibitem{Delange:1975:chiffres}
\bysame, \emph{Sur la fonction sommatoire de la fonction ``somme des
  chiffres''}, Enseignement Math. (2) \textbf{21} (1975), 31--47.

\bibitem{Drmota:2000}
Michael Drmota, \emph{The distribution of patterns in digital expansions},
  Algebraic Number Theory and Diophantine Analysis (F.~Halter-Koch and R.~F.
  Tichy, eds.), de Gruyter (Berlin), 2000, pp.~103--121.

\bibitem{Dumas:2014:asymp}
Philippe Dumas, \emph{Asymptotic expansions for linear homogeneous
  divide-and-conquer recurrences: Algebraic and analytic approaches collated},
  Theoret. Comput. Sci. \textbf{548} (2014), 25--53.

\bibitem{Flajolet-Ramshaw:1980:gray}
Philippe Flajolet and Lyle Ramshaw, \emph{A note on {G}ray code and odd-even
  merge}, SIAM J. Comput. \textbf{9} (1980), 142--158.

\bibitem{Flajolet-Sedgewick:ta:analy}
Philippe Flajolet and Robert Sedgewick, \emph{Analytic combinatorics},
  Cambridge University Press, Cambridge, 2009.

\bibitem{Gelfond:1968:sur}
A.~O. Gel$'$fond, \emph{Sur les nombres qui ont des propri\'et\'es additives et
  multiplicatives donn\'ees}, Acta Arith. \textbf{13} (1967/1968), 259--265.

\bibitem{Grabner-Heuberger:2006:Number-Optimal}
Peter~J. Grabner and Clemens Heuberger, \emph{On the number of optimal base 2
  representations of integers}, Des. Codes Cryptogr. \textbf{40} (2006), no.~1,
  25--39.

\bibitem{Heuberger-Kropf:2013:analy}
Clemens Heuberger and Sara Kropf, \emph{Analysis of the binary asymmetric joint
  sparse form}, Combin. Probab. Comput. \textbf{23} (2014), 1087--1113.

\bibitem{Heuberger-Kropf-Prodinger:2015:output}
Clemens Heuberger, Sara Kropf, and Helmut Prodinger, \emph{Output sum of
  transducers: Limiting distribution and periodic fluctuation}, Electron. J.
  Combin. \textbf{22} (2015), no.~2, 1--53.

\bibitem{Hwang:1998}
Hsien-Kuei Hwang, \emph{On convergence rates in the central limit theorems for
  combinatorial structures}, European J. Combin. \textbf{19} (1998), 329--343.

\bibitem{Joye-Yen:2000:optim-left}
Marc Joye and Sung-Ming Yen, \emph{Optimal left-to-right binary signed digit
  recoding}, IEEE Trans. Comput. \textbf{49} (2000), no.~7, 740--748.

\bibitem{Kirschenhofer:1983:subbl}
Peter Kirschenhofer, \emph{Subblock occurrences in the {$q$}-ary representation
  of {$n$}}, SIAM J. Algebraic Discrete Methods \textbf{4} (1983), no.~2,
  231--236.

\bibitem{Kropf-Wagner:2016:quasiad-funct}
Sara Kropf and Stephan Wagner, \emph{$q$-quasiadditive functions}, Proceedings
  of the 27th International Conference on Probabilistic, Combinatorial and
  Asymptotic Methods for the Analysis of Algorithms, 2016, arXiv:1605.03654
  [math.CO].

\bibitem{OEIS:2016}
\emph{The {O}n-{L}ine {E}ncyclopedia of {I}nteger {S}equences},
  \url{http://oeis.org}, 2016.

\bibitem{Reitwiesner:1960}
George~W. Reitwiesner, \emph{Binary arithmetic}, Advances in Computers, {V}ol.
  1, Academic Press, New York, 1960, pp.~231--308.

\bibitem{Sloane:number-on}
Neil J.~A. Sloane, \emph{The number of {ON} cells in cellular automata},
  arXiv:1503.01168 [math.CO].

\bibitem{Thuswaldner:1999}
J{\"o}rg~M. Thuswaldner, \emph{Summatory functions of digital sums occurring in
  cryptography}, Period. Math. Hungar. \textbf{38} (1999), no.~1-2, 111--130.

\end{thebibliography}

\end{document}